\documentclass[12pt]{amsart}
\usepackage[english]{babel}
\usepackage{graphicx}
\usepackage[centertags]{amsmath}
\usepackage{amsfonts}
\usepackage{amssymb}
\usepackage{amsthm}
\usepackage{newlfont}
\usepackage{color}
\newtheorem{thm}{Theorem}[section]
\newtheorem{cor}[thm]{Corollary}

\newtheorem{prop}[thm]{Proposition}
\theoremstyle{definition}

\theoremstyle{remark}
\newtheorem{rem}[thm]{Remark}
\newcommand{\id}{\textrm{id}}

\begin{document}
\title{Curvature properties of Riemannian manifolds with skew-circulant structures}
\bigskip
\author{Iva Dokuzova}
\date{}

\begin{abstract}
 We consider a $4$-dimensional Riemannian manifold $M$ endowed with a right skew-circulant tensor structure $S$, which is an isometry with respect to the metric $g$ and the fourth power of $S$ is minus identity. We determine a class of manifolds $(M, g, S)$, whose curvature tensors are invariant under $S$. For such manifolds we obtain properties of the Ricci tensor. Also we get expressions of the sectional curvatures of some special $2$-planes in a tangent space of $(M, g, S)$.
\end{abstract}
\maketitle
\Small{\textbf{Mathematics Subject Classification (2020)}:  53B20, 53C25, 53C15, 53C55.}

\Small{\textbf{Keywords}:  Riemannian manifold; sectional curvature; skew-circulant matrix; Ricci curvature.
 \section{Introduction}
\label{intro}

Almost Hermitian manifolds were classified with respect to the covariant derivative of the almost complex structure $J$ by A.~Gray and L.~Hervella (\cite{GrHer}).  The Hermitian manifolds form a class of manifolds with an integrable structure $J$. Part of these manifolds are the well-known and extensively explored K\"{a}hler manifolds. The locally conformal K\"{a}hler manifolds are Hermitian manifolds whose metric is conformal to a K\"{a}hler metric in a neighborhood at each point. These manifolds are studied by many geometers (for example \cite{angella, huang, Khuz, moroianu-ornea, prvanovic, vilcu}).

In \cite{gray}, there are introduced three classes of almost Hermitian manifolds, determined by Gray's curvature identities, and it is proved that every K\"{a}hler manifold satisfies these identities. Many investigations of the geometry of the almost Hermitian manifolds are devoted to the study of the curvature tensors, the Ricci tensors, the scalar curvatures and the sectional curvatures of some characteristic 2-planes of the tangent spaces of the manifolds.  We will mention some papers that consider classes of manifolds, whose curvature tensors are invariant under the almost complex structure $J$ (\cite{Sca-Vezz, gad-mont, gray-van, prvn, raz-dzhe, van, yu}).

The present paper is a continuation of our research made in \cite{dok-raz} and \cite{raz-dok}. We study a 4-dimensional Riemannian manifold $M$ with a metric $g$ and equipped with a tensor $S$ of type $(1,1)$, compatible with $g$. The tensor structure $S$ satisfies $S^{4}= -\id$ and its local components form a right skew-circulant matrix. We note that such a manifold $(M, g, S)$ is associated with a locally conformal K\"{a}hler manifold $(M, g, J=S^{2})$ (\cite{raz-dok}). Our purpose is to find geometrical properties of manifolds $(M, g, S)$ that satisfy analogues of Gray's curvature identities. Thus, we study a class of manifolds whose curvature tensors are invariant under $S$.

The paper is organized as follows. In Section~\ref{sec:1}, we give some necessary facts about a 4-dimensional Riemannian manifold $(M, g, S)$ obtained in \cite{dok-raz}. Also, we determine a class of manifolds $(M, g, S)$ whose curvature tensors are invariant under $S$. The following results in Section~\ref{sec:2} and Section~\ref{sec:3} refer to such manifolds. In Section~\ref{sec:2}, we express the Ricci tensor of $(M, g, S)$ by the metric $g$ and by the additional structure $S$. We establish that $(M, g, S)$ is an almost Einstein manifold and get a condition under which it is an Einstein manifold. We find the Ricci curvatures in the direction of the basis vectors $x, Sx, S^{2}x, S^{3}x$ in a tangent space $T_{p}M$.
 In Section~\ref{sec:3}, we obtain relations between the sectional curvatures of the $2$-planes determined by the $S$-basis $\{S^{3}x, S^{2}x, Sx, x\}$. Also, we get expressions of these sectional curvatures by the angles between vectors.

\section{Preliminaries}\label{sec:1}

Let $M$ be a $4$-dimensional differentiable manifold with a metric $g$. Let $S$ be a tensor of type $(1,1)$ on $M$, with local coordinates given by the right skew-circulant matrix
\begin{equation}\label{S}
    (S_{i}^{j})=\begin{pmatrix}
      0 & 1 & 0 & 0\\
      0 & 0 & 1 & 0 \\
      0 & 0 & 0 & 1\\
      -1 & 0 & 0 & 0\\
    \end{pmatrix}.
\end{equation}
Then
\begin{equation}\label{q4}
    S^{4}=-\id.
\end{equation}
We suppose that the structure $S$ of the manifold $M$ acts as an isometry with respect to $g$, i.e.,
\begin{equation}\label{2.1}
    g(Sx, Sy)=g(x, y).
\end{equation}
Anywhere in this work $x, y, z, u, w$ will stand for arbitrary elements of the algebra of the smooth vector fields on $M$ or vectors in the tangent space $T_{p}M$ at a point $p\in M$. The Einstein summation convention is used, the range of the summation indices being always $\{1, 2, 3, 4\}$.

The equalities \eqref{S} and \eqref{2.1} imply the following right skew-circulant matrix of components of $g$:
\begin{equation}\label{metricg}
    (g_{ij})=\begin{pmatrix}
      A & B & 0 & -B \\
      B & A & B & 0\\
      0 & B & A & B\\
      -B & 0 & B & A\\
    \end{pmatrix}.
\end{equation}
In \eqref{metricg}, functions $A$ and $B$ of four variables  $x^{1}$, $x^{2}$,
$x^{3}$ and $x^{4}$ are smooth on $M$.
It is supposed that $A >\sqrt{2} B >0$ in order $g$ to be positive definite. The manifold $(M, g, S)$ is introduced in \cite{dok-raz}.

On $(M, g, S)$ we consider an associated metric $\tilde{g}$ with $g$, determined by \begin{equation}\label{metricf}
 \tilde{g}(x, y)=g(x, Sy)+g(Sx, y).\end{equation}
 In \cite{dok-raz}, it is verified that $\tilde{g}$ is a necessary indefinite.
The matrix of its components is
\begin{equation}\label{g-tilde}
(\tilde{g}_{ij})=\begin{pmatrix}
      2B & A & 0 & -A \\
      A & 2B & A &0 \\
      0 & A & 2B & A\\
      -A & 0 & A & 2B\\
    \end{pmatrix}.
\end{equation}

Let $\nabla$ be the Riemannian connection of $g$. The curvature tensor $R$ of $\nabla$ is determined by
$R(x, y)z=\nabla_{x}\nabla_{y}z-\nabla_{y}\nabla_{x}z-\nabla_{[x,y]}z.$
The tensor of type $(0, 4)$  associated with $R$ is defined by
$R(x, y, z, u)=g(R(x, y)z,u).$

In \cite{dok-raz}, it is determined a class of manifolds $(M, g, S)$ which satisfy the identity
\begin{equation}\label{R1}
  R(x, y, Sz, Su)=R(x, y, z, u).
\end{equation}
 The above property defines a more general class of manifolds than this one with $\nabla S=0$.

On the other hand, equality \eqref{R1} and the symmetries of $R$ imply
\begin{equation}\label{R}
  R(Sx, Sy, Sz, Su)=R(x, y, z, u).
\end{equation}
Therefore, the class of manifolds $(M, g, S)$ with the condition \eqref{R} is more general than the class $(M, g, S)$ with the condition \eqref{R1}.
\begin{rem}
Every manifold $(M, g, S)$ is associated with an almost Hermitian manifold $(M, g, J)$, where $J=S^{2}$. Moreover, the manifold $(M, g, J)$ is locally conformal K\"{a}hler (\cite{raz-dok}).
For almost Hermitian manifolds, Grey's classification is valid (\cite{gray}). It is made with respect to the curvature tensor $R$. Due to this classification the class $\mathcal{L}_{3}$ consists of manifolds which satisfy the identity $R(Jx, Jy, Jz, Ju)=R(x, y, z, u)$.

It is easy to see that if $(M, g, S)$ has the property \eqref{R}, then the associated manifold $(M, g, J)$ belongs to $\mathcal{L}_{3}$.
\end{rem}

\section{Almost Einstein manifolds}\label{sec:2}

The Ricci tensor $\rho$ with respect to $g$ is given by the well-known formula
$\rho(y,z)=g^{ij}R(e_{i}, y, z, e_{j}).$
The scalar curvature $\tau$ with respect to $g$ and its associated quantity $\tau^{*}$ are determined by
    $\tau=g^{ij}\rho(e_{i}, e_{j}),\ \tau^{*}=\tilde{g}^{ij}\rho(e_{i}, e_{j}).$
Here $\{e_{i}\}$ is a local basis of $T_{p}M$, $g^{ij}$ and $\tilde{g}^{ij}$ are the components of the inverse matrices of $g$ and $\tilde{g}$, respectively.


By $R_{ijkh}$ and $\rho_{ij}$ we will denote the components of the curvature tensor $R$ and the components of the Ricci tensor $\rho$ with respect to $\{e_{i}\}$.
Hence, we establish the following propositions.
\begin{prop}\label{con-l1ll}
The property \eqref{R} of the curvature tensor $R$ of the manifold $(M, g, S)$ is equivalent to the conditions
\begin{align}\label{L2}\nonumber
     R_{1313}=&R_{2424},\ R_{1212}=R_{1414}=R_{2323}=R_{3434},\\
R_{1213}=&R_{2414}=R_{2423}=R_{2313}=R_{1334},\ R_{1324}=2R_{1234}=2R_{1423},\\\nonumber R_{1223}=&R_{1214}=R_{1434}=R_{2334},\ R_{1224}=R_{1413}=R_{2434}.
\end{align}
\end{prop}
\begin{proof}
The local form of \eqref{R} is
\begin{equation}\label{localL1}
R_{abcd}S_{i}^{a}S_{j}^{b}S_{k}^{c}S_{h}^{d}=R_{ijkh}.\end{equation}
Then, using \eqref{S} and applying the Bianchi first identity to the components of $R$, we obtain \eqref{L2}.

Vice versa, from \eqref{S} and \eqref{L2} it follows \eqref{localL1}, so \eqref{R} holds true.
\end{proof}

\begin{prop}\label{the1}
If a manifold $(M, g, S)$ has the property \eqref{R}, then the components of the Ricci tensor $\rho$ satisfy
\begin{equation}\label{system-rho}
\rho_{11}=\rho_{22}=\rho_{33}=\rho_{44},\quad \rho_{12}=\rho_{23}=\rho_{34}=-\rho_{14},\quad \rho_{13}=\rho_{24}=0.
\end{equation}
\end{prop}
\begin{proof}
Due to Proposition~\ref{con-l1ll}, the components of $R$ satisfy \eqref{L2}.
For brevity, we denote
\begin{equation}\label{RV2}
\begin{split}
R_{1}=R_{1313},\  R_{2}=R_{1212},\ R_{3}=R_{1223},\ R_{4}=R_{1234},\  R_{5}=R_{1213},\ R_{6}=R_{1224}.
\end{split}
\end{equation}

The inverse matrix of $(g_{ij})$ is
\begin{equation}\label{g-obr}
    (g^{ij})=\frac{1}{A^{2}-2B^{2}}\begin{pmatrix}
      A & -B & 0 & B \\
      -B & A & -B & 0\\
      0 & -B & A & -B \\
      B & 0 & -B & A\\
    \end{pmatrix}.
\end{equation}
Now, having in mind \eqref{L2}, \eqref{RV2} and \eqref{g-obr}, we calculate the components of the Ricci tensor $\rho$ as follows:
  \begin{equation*}
  \begin{split}
      \rho_{11}&=\rho_{22}=\rho_{33}=\rho_{44}=\frac{1}{A^{2}-2B^{2}}\big(2B(R_{5}+R_{6})-A(2R_{2}+R_{1})\big),\\
    \rho_{12}&=\rho_{23}=\rho_{34}=-\rho_{14}=\frac{1}{A^{2}-2B^{2}}\big(B(2R_{3}-R_{2}+3R_{4})-A(R_{5}+R_{6})\big),\\
    \rho_{13}&=\rho_{24}=0.
    \end{split}
\end{equation*}
So the equalities \eqref{system-rho} are valid.
\end{proof}

A Riemannian manifold is said to be Einstein if its Ricci tensor $\rho$ is a constant multiple of the metric tensor $g$, i.e.
\begin{equation}\label{E}\rho(x, y) = \alpha g(x, y).\end{equation}

In \cite{dok-raz}, a Riemannian manifold $(M, g, S)$ is called
almost Einstein if its Ricci tensor $\rho$, the metrics $g$ and $\tilde{g}$ satisfy
\begin{equation}\label{AE}\rho(x, y) = \alpha g(x, y) + \beta \tilde{g}(x, y),\end{equation} where $\alpha$ and $\beta$ are smooth functions on $M$.

 \begin{thm}\label{prop5}
The manifold $(M, g, S)$ with property \eqref{R} is almost Einstein.
\end{thm}
\begin{proof}
The inverse matrix of $(\tilde{g}_{ij})$ is
\begin{equation}\label{f-obr}
    (\tilde{g}^{ij})=\frac{1}{2(A^{2}-2B^{2})}\begin{pmatrix}
      -2B & A & 0 & -A \\
      A & -2B & A &0 \\
      0 & A & -2B & A\\
      -A & 0 & A & -2B\\
    \end{pmatrix}.
\end{equation}
According to Proposition~\ref{the1}, for $(M, g, S)$ the equalities \eqref{system-rho} are valid.
Consequently, using \eqref{system-rho}, \eqref{g-obr} and \eqref{f-obr} we calculate the values of the scalar curvature $\tau$ and the associated quantity $\tau^{*}$ as follows:
  \begin{equation*}
    \tau =\frac{4}{A^{2}-2B^{2}}\big(A\rho_{11}-2B\rho_{12}\big),\quad
        \tau^{*} =\frac{4}{A^{2}-2B^{2}}\big(A\rho_{12}-B\rho_{11}\big).
\end{equation*}
Immediately from \eqref{system-rho} and the latter equalities we have
\begin{equation}\label{rho-tau}
    \rho_{11}=\frac{\tau}{4}A+\frac{2\tau^{*}}{4}B,\quad
     \rho_{12}=\frac{\tau}{4}B+\frac{\tau^{*}}{4}A,\quad \rho_{13}=\frac{\tau}{4}0+\frac{\tau^{*}}{4}0,
\end{equation}
and bearing in mind \eqref{metricg} and \eqref{g-tilde} we get
\begin{equation*}
    \rho_{11}=\frac{\tau}{4}g_{11}+\frac{\tau^{*}}{4}\tilde{g}_{11},\quad
     \rho_{12}=\frac{\tau}{4}g_{12}+\frac{\tau^{*}}{4}\tilde{g}_{12}, \quad
     \rho_{13}=\frac{\tau}{4}g_{13}+\frac{\tau^{*}}{4}\tilde{g}_{13}.
\end{equation*}
Then, taking into account \eqref{metricg}, \eqref{g-tilde}, \eqref{system-rho} and \eqref{rho-tau}, we obtain
\begin{equation*}
     \rho_{ij}=\frac{\tau}{4}g_{ij}+\frac{\tau^{*}}{4}\tilde{g}_{ij},
\end{equation*}
i.e.
\begin{equation}\label{rho51}
     \rho(x,y)=\frac{\tau}{4}g(x,y)+\frac{\tau^{*}}{4}\tilde{g}(x,y).
\end{equation}
Therefore, comparing \eqref{rho51} with \eqref{AE}, we state that $(M, g, S)$ is an almost Einstein manifold.
\end{proof}

Let $(M, g, S)$ satisfy the conditions of Theorem~\ref{prop5}. If we suppose that $(M, g, S)$ is an Einstein manifold, then its Ricci tensor $\rho$ has the form \eqref{E}. Hence \eqref{rho51} yields the following
\begin{cor}\label{prop-enstein}
If the manifold $(M, g, S)$ with property \eqref{R} is Einstein then
\begin{equation*}
    \tau^{*}=0.
 \end{equation*}
\end{cor}
\subsection{Ricci curvatures}
Now, we recall that the Ricci curvature in the direction of a nonzero vector $x$ is the value \begin{equation}\label{Ricicurv}
    r(x)=\frac{\rho(x,x)}{g(x,x)}.
\end{equation}

 Let $x$ be an arbitrary nonzero vector in a tangent space $T_{p}M$ of $(M, g, S)$.
 The vectors $x$, $Sx$, $S^{2}x$ and $S^{3}x$ determine six angles which belong to $(0, \pi)$.
For these angles, the next statements are verified in \cite{dok-raz}. Namely, if $\varphi=\angle(x, Sx)$ then \begin{equation}\label{ugli}
\begin{split}
 \angle(x, Sx)=&\angle(Sx, S^{2}x)=\angle(S^{2}x, S^{3}x)=\varphi,\quad \angle(x, S^{3}x)=\pi-\varphi,\\ \angle(x, S^{2}x)=&\angle(Sx, S^{3}x)=\frac{\pi}{2},
 \end{split}
\end{equation}
\begin{equation}\label{inequalities}
  \frac{\pi}{4}<\varphi<\frac{3\pi}{4}.
\end{equation}


A basis of type $\{S^{3}x, S^{2}x, Sx, x\}$  of $T_{p}M$ is called an $S$-\textit{basis}. In this case we say that \textit{the vector $x$ induces an $S$-basis of} $T_{p}M$.
 The existence of an orthogonal $S$-basis of $T_{p}M$ is based on \eqref{ugli} and \eqref{inequalities}.
 \begin{thm}\label{th-ricci}
Let $(M, g, S)$ be a manifold with property \eqref{R}. If a vector $x$ induces an $S$-basis, then the Ricci curvatures in the direction of the basis vectors are
\begin{equation}\label{ricc}
     r(x)=r(Sx)=r(S^{2}x)=r(S^{3}x)=\frac{\tau}{4}+\frac{\tau^{*}}{2}\cos\varphi,
\end{equation}
where $\angle(x, Sx)=\varphi$.
\end{thm}
\begin{proof}
 In the course of the proof of Theorem~\ref{prop5}, we find that $\rho$ is given by \eqref{rho51}. Therefore, using \eqref{2.1} and \eqref{metricf}, we obtain
\begin{equation}\label{rhoL2}
\begin{split}
    \rho(x, x)=\rho(Sx, Sx)=\rho(S^{2}x, S^{2}x)&=\rho(S^{3}x, S^{3}x)=\frac{\tau}{4}g(x,x)+\frac{\tau^{*}}{4}\tilde{g}(x,x).
    \end{split}
\end{equation}

Let a vector $x$ induce an $S$-basis. The equalities \eqref{2.1} and \eqref{metricf} imply
\begin{equation*}
\begin{split}
g(x, Sx)=g(x, x)\cos\varphi,\qquad \tilde{g}(x, x)=2g(x, x)\cos\varphi,
  \end{split}
 \end{equation*}
and using \eqref{Ricicurv} and \eqref{rhoL2} we find \eqref{ricc}.
\end{proof}

\begin{cor}
Let $(M, g, S)$ with \eqref{R} be an Einstein manifold. If a vector $x$ induces an $S$-basis, then the Ricci curvatures in the direction of the basis vectors are
\begin{equation*}
     r(x)=r(Sx)=r(S^{2}x)=r(S^{3}x)=\frac{\tau}{4}.
\end{equation*}
\end{cor}
\begin{proof}
The above equalities follow directly by substituting $\tau^{*}=0$ into \eqref{ricc}.
\end{proof}

\subsection{An example}

 Examples of Einstein and almost Einstein manifolds $(M, g, S)$ with condition \eqref{R1} are constructed on Lie groups in \cite{dok-raz}. Another example of an Einstein manifold $(M, g, S)$ is constructed on a Lie group in \cite{raz-dok}. Since this example is not specified as a manifold with identity \eqref{R}, we will now focus only on this property of it.

Let $G$ be a $4$-dimensional real connected Lie group. Let $\mathfrak{g}$ be the corresponding Lie algebra with a basis $\{e_{1}, e_{2},e_{3},e_{4}\}$ of left invariant vector fields. We introduce a skew-circulant structure $S$ and a metric $g$ as follows:
\begin{eqnarray*}
&Se_{1}=-e_{4},\quad Se_{2}=e_{1},\quad Se_{3}=e_{2},\quad Se_{4}=e_{3},\label{lie} \\
&g(e_{i}, e_{j})= \delta_{ij},\label{g}
\end{eqnarray*}
where $\delta_{ij}$ is the Kronecker delta.

The used basis $\{e_i\}$ is an orthonormal $S$-basis. The equalities \eqref{q4} and \eqref{2.1} are valid and $(g, S)$ is a structure of the considered type. We denote the corresponding manifold by $(G, g, S)$.

The real four-dimensional indecomposable Lie algebras are classified by Mubarak\-zyanov (\cite{mub}). More accessible in the libraries is the paper \cite{Biggs} and the references therein.
We consider the class $\{\mathfrak{g}_{4,5}\}$, which represents an indecomposable Lie algebra,
depending on two real parameters $a$ and $b$.

According to the definition of the class $\{\mathfrak{g}_{4,5}\}$, we have that the nonzero brackets are as follows (\cite{Biggs}):
\begin{equation*}
  [e_{1}, e_{4}]=e_{1},\quad [e_{2}, e_{4}]=ae_{2},\quad [e_{3}, e_{4}]=be_{3},\quad
  -1\leq b \leq a \leq 1,\ ab\neq 0.
\end{equation*}
The components of $\nabla$ are
\begin{equation*}
\begin{array}{lll}
    \nabla_{e_{1}}e_{1}=-e_{4},\quad
&    \nabla_{e_{1}}e_{4}=e_{1},\quad
&    \nabla_{e_{2}}e_{2}=-ae_{4},\\
   \nabla_{e_{2}}e_{4}=ae_{2},\quad
&    \nabla_{e_{3}}e_{4}=be_{3},\quad
&    \nabla_{e_{3}}e_{3}=-be_{4}.
\end{array}
\end{equation*}

In \cite{raz-dok}, for a manifold $(G, g, S)$ with a Lie algebra $\mathfrak{g}$ from the class $\{\mathfrak{g}_{4, 5}\}$, we get that if $a=b=1$ then

i) the nonzero components of the curvature tensor $R$ and of the Ricci tensor $\rho$ with respect to $\{e_i\}$, are as follows:
\begin{align}\label{rlamda}
  R_{1212}=R_{1414}=R_{2323}=R_{3434}=R_{1313}=R_{2424}=1,
\end{align}
\begin{equation*}
    \rho_{11}=\rho_{22}=\rho_{33}=\rho_{44}=-3.
\end{equation*}

 ii) $(G, g, S)$ is a non-flat Einstein manifold with a negative scalar curvature $\tau=-12$.

Bearing in mind the above results we state the following
\begin{prop}
 Let $(G, g, S)$ be a manifold with a Lie algebra $\mathfrak{g}$ from the class $\{\mathfrak{g}_{4, 5}\}$. If $a=b=1$ are valid, then $(G, g, S)$ satisfies the property \eqref{R}.
\end{prop}
\begin{proof}
Directly from \eqref{rlamda} we get that the components of the curvature tensor satisfy \eqref{L2}.
\end{proof}
So we obtain an example of an Einstein manifold with the property \eqref{R}. Also, this manifold does not satisfy \eqref{R1}.

\section{Some sectional curvatures}\label{sec:3}

If $\{x, y\}$ is a non-degenerate $2$-plane spanned by vectors $x, y$ in $T_{p}M$, then its sectional curvature is
\begin{equation}\label{3.3}
    k(x,y)=\frac{R(x, y, x, y)}{g(x, x)g(y, y)-g^{2}(x, y)}.
\end{equation}
Let a vector $x$ induce an $S$-basis in $T_{p}M$ of $(M,  g, S)$. There are determined six 2-planes $\{x, Sx\}$, $\{x, S^{2}x\}$, $\{x, S^{3}x\}$, $\{Sx, S^{2}x\}$, $\{Sx, S^{3}x\}$ and $\{S^{2}x, S^{3}x\}$ of $T_{p}M$. For the angles between the pairs of vectors equalities \eqref{ugli} are valid. Moreover, the angle $\varphi=\angle(x, Sx)$ satisfies \eqref{inequalities}.

In the next theorems we establish relations among the sectional curvatures of the 2-planes generated by an $S$-basis.
\begin{thm}\label{predl}
Let $(M, g, S)$ be a manifold with property \eqref{R}. If a vector $x$ induces an $S$-basis, then for the sectional curvatures of the basic $2$-planes we have
\begin{equation}\label{3.21}
k(x,Sx)=k(Sx,S^{2}x)=k(S^{2}x,S^{3}x)=k(S^{3}x,x),
\end{equation}
\begin{equation}\label{3.31}
k(x,S^{2}x)=k(Sx,S^{3}x).
\end{equation}
\end{thm}
\begin{proof}
From (\ref{R}) we obtain
\begin{equation}\label{Rv1}
\begin{split}
R(x,y,z,u)=&R(Sx,Sy,Sz,Su)\\=&R(S^{2}x,S^{2}y,S^{2}z,S^{2}u)=R(S^{3}x,S^{3}y,S^{3}z,S^{3}u).
\end{split}
\end{equation}
In \eqref{Rv1} we substitute

1) $Sx$ for $y$, $x$ for $z$, $Sx$ for $u$, and we get
 \begin{equation}\label{slRv1}
 \begin{split}
 R(x,Sx,x,Sx)=&R(Sx,S^{2}x,Sx,S^{2}x )\\=&R(S^{2}x,S^{3}x,S^{2}x,S^{3}x)=R(S^{3}x,x,S^{3}x,x).
 \end{split}
 \end{equation}
2) $S^{2}x$ for $y$, $x$ for $z$, $S^{2}x$ for $u$, and we find
 \begin{equation}\label{dop4}
 R(x,S^{2}x,x,S^{2}x)=R(Sx,S^{3}x,Sx,S^{3}x).
 \end{equation}
 The equality \eqref{3.21} follows from \eqref{2.1}, \eqref{3.3} and \eqref{slRv1}. In a similar way, from \eqref{2.1}, \eqref{3.3} and \eqref{dop4} we obtain \eqref{3.31}.
\end{proof}

Due to Theorem~\ref{predl} there are only two different sectional curvatures determined by an $S$-basis. Thus, we consider only the sectional curvatures of the type $k(x, Sx)$ and $k(x, S^{2}x)$.
\begin{thm}\label{th4}
Let $(M, g, S)$ be a manifold with property \eqref{R}. If vectors $x$ and $u$ induce $S$-bases and $\{S^{3}x, S^{2}x, Sx, x\}$ is an orthonormal basis, then the sectional curvatures satisfy
\begin{equation}\label{mu+mu}
    \begin{split}
   2(1-&\cos^{2}\varphi)k(u,Su)+k(u,S^{2}u)=k(x,S^{2}x) +2(1-\cos^{2}\varphi)k(x,Sx)\\&+\cos^{2}\varphi \big(4R(x, Sx, Sx, S^{2}x)+6R(x, Sx, S^{2}x, S^{3}x)\big)\\&+4\cos\varphi\big(R(x, Sx, x, S^{2}x)+R(x, Sx, Sx, S^{3}x)\big).
    \end{split}
\end{equation}
where $\angle(u, Su)=\varphi$.
\end{thm}
\begin{proof}
In \eqref{Rv1} we substitute

    1)\ $Sx$ for $y$, $S^{2}x$ for $z$ and $x$ for $u$, and we get
    \begin{equation}\label{dop3}
    \begin{split}
    R(x, Sx, S^{2}x, x)=&R(Sx, S^{2}x, S^{3}x, Sx)\\=&-R(S^{2}x, S^{3}x, x, S^{2}x)=R(S^{3}x, x, Sx, S^{3}x),
    \end{split}
    \end{equation}
    2)\  $Sx$ for $y$, $Sx$ for $z$ and $S^{2}x$ for $u$, and we have
    \begin{equation}\label{dop5}
   \begin{split}
    R(x, Sx, Sx, S^{2}x)=&R(Sx, S^{2}x, S^{2}x, S^{3}x)\\=&-R(S^{2}x, S^{3}x,S^{3} x, x)=-R(S^{3}x, x, x, Sx),
    \end{split}
    \end{equation}
    3)\  $Sx$ for $y$, $S^{2}x$ for $z$ and $S^{3}x$ for $u$, and we find
    \begin{equation}\label{dop6}
  \begin{split}
    R(x, Sx,  S^{2}x, S^{3}x)=&-R(Sx, S^{2}x, S^{3}x, x)\\=&R(S^{2}x, S^{3}x, x, Sx)=-R(S^{3}x, x, Sx,  S^{2}x),
    \end{split}
    \end{equation}
    4)\ $Sx$ for $y$, $Sx$ for $z$ and $S^{3}x$ for $u$, and we get
    \begin{equation}\label{dop7}
     \begin{split}
    R(x, Sx,  Sx, S^{3}x)=&-R(Sx, S^{2}x, S^{2}x, x)\\=&-R(S^{2}x, S^{3}x, S^{3}x, Sx)=-R(S^{3}x, x, x,  S^{2}x),
    \end{split}
    \end{equation}
    5)\  $S^{2}x$ for $y$, $Sx$ for $z$ and $S^{3}x$ for $u$, and we have
    \begin{equation}\label{dop8}
     \begin{split}
    R(x, S^{2}x,  Sx, S^{3}x)=&-R(Sx, S^{3}x, S^{2}x, x)\\=&R(S^{2}x, x, S^{3}x, Sx)=-R(S^{3}x, Sx, x,  S^{2}x).
    \end{split}
    \end{equation}

 Let $u=\alpha S^{3}x+\beta S^{2}x +\gamma S x+\delta x$, where $\alpha,\beta,\gamma, \delta \in \mathbb{R}$. From \eqref{S} we get $Su=\beta S^{3}x+\gamma S^{2}x +\delta Sx - \alpha x$, $S^{2}u=\gamma S^{3}x+\delta S^{2}x -\alpha Sx -\beta x$ and $S^{3}u= \delta S^{3}x-\alpha S^{2}x -\beta Sx - \gamma x$.
 Then, by using the linear properties of the curvature tensor $R$ and having in mind \eqref{Rv1}, \eqref{slRv1}, \eqref{dop4}, \eqref{dop3} -- \eqref{dop8}, we obtain
 \begin{align*}
    R(u,Su,u,Su)&=(\alpha^{2}+\gamma^{2})(\beta^{2}+\delta^{2})R_{1}\\&+\Big((\alpha^{2}+\gamma^{2})^{2}+(\beta^{2}+\delta^{2})^{2}+2\delta\beta(\alpha^{2}-\gamma^{2})+2\alpha\gamma(\delta^{2}-\beta^{2})\Big)R_{2}\\&+2(\alpha^{2}+\gamma^{2})(\beta^{2}+\delta^{2})R_{3}\\&+6\Big(\delta\beta(\gamma^{2}-\alpha^{2})+\alpha\gamma(\beta^{2}-\delta^{2})\Big)R_{4}\\&+2\Big((\beta\gamma-\alpha\delta)(\beta^{2}+\delta^{2})+(\alpha\beta+\delta\gamma)(\alpha^{2}+\gamma^{2})\Big)R_{5}\\&+2\Big((\alpha\beta+\gamma\delta)(\delta^{2}+\beta^{2})+(\beta\gamma-\delta\alpha)(\alpha^{2}+\gamma^{2})\Big)R_{6},\\ R(u,S^{2}u,u,S^{2}u)&=\Big((\alpha^{2}+\gamma^{2})^{2}+(\beta^{2}+\delta^{2})^{2}\Big)R_{1}\\&+2(\alpha^{2}+\gamma^{2})(\beta^{2}+\delta^{2})R_{2}\\&+8\Big(\delta\beta(\gamma^{2}-\alpha^{2})+\alpha\gamma(\beta^{2}-\delta^{2})\Big)R_{3}\\&+6(\alpha^{2}+\gamma^{2})(\beta^{2}+\delta^{2})R_{4}\\&+4\Big((\alpha\beta+\gamma\delta)(\delta^{2}+\beta^{2})+(\beta\gamma-\delta\alpha)(\alpha^{2}+\gamma^{2}))R_{5}\\&+4\Big((\beta\gamma-\alpha\delta)(\beta^{2}+\delta^{2})+(\alpha\beta+\delta\gamma)(\alpha^{2}+\gamma^{2})\Big)R_{6},
\end{align*}
where
\begin{equation}\label{r1-6}
    \begin{array}{lll}
R_{1}=R(x, S^{2}x, x, S^{2}x), & R_{2}=R(x, Sx, x, Sx), & R_{3}=R(x, Sx, Sx, S^{2}x),\\
R_{4}=R(x, Sx, S^{2}x, S^{3}x),& R_{5}=R(x, Sx, x, S^{2}x),& R_{6}=R(x, Sx, Sx, S^{3}x).
\end{array}
\end{equation}
Now we get the sum
\begin{equation}\label{r+k}
    \begin{split}
   2R(u,Su,u,Su)&+R(u,S^{2}u,u,S^{2}u) =(\alpha^{2}+\beta^{2}+\gamma^{2}+\delta^{2})^{2}R_{1}\\&+2\big((\alpha^{2}+\beta^{2}+\gamma^{2}+\delta^{2})^{2}-(\alpha\beta+\beta\gamma+\delta\gamma-\alpha\delta)^{2}\big) R_{2}\\&+(\alpha\beta+\beta\gamma+\delta\gamma-\alpha\delta)^{2}(4R_{3}+ 6R_{4})\\&+ 4(\alpha\beta+\beta\gamma+\delta\gamma-\alpha\delta)(R_{5}+ R_{6}).
    \end{split}
\end{equation}

Since the basis $\{S^{3}x, S^{2}x, Sx, x\}$ is orthonormal, we have
\begin{equation*}
    g(u, u)= g(Su, Su)=\alpha^{2}+\beta^{2}+\gamma^{2}+\delta^{2},
     \end{equation*}
     \begin{equation*}
g(u,Su)= \alpha\beta+\beta\gamma+\delta\gamma-\alpha\delta,\quad g(u, S^{2}u)=0.
\end{equation*}
We assume that $u$ is a unit vector and from the latter equalities we get
\begin{equation}\label{cos-alfa}
    \alpha^{2}+\beta^{2}+\gamma^{2}+\delta^{2}=1,\quad \alpha\beta+\beta\gamma+\delta\gamma-\alpha\delta=\cos\varphi,
\end{equation}
where $\varphi=\angle(u, Su)$.
Using \eqref{r+k} and \eqref{cos-alfa} we calculate
\begin{equation}\label{r+r}
    \begin{split}
   2R(u,Su,u,Su)+&R(u,S^{2}u,u,S^{2}u)= R_{1}+ 2(1-\cos^{2}\varphi) R_{2}\\+&\cos^{2}\varphi (4R_{3}+6 R_{4})+ 4\cos\varphi(R_{5}+ R_{6}).
    \end{split}
\end{equation}

Bearing in mind that $x$ induces an orthonormal $S$-basis and $u$ is a unit vector, with the help of \eqref{2.1}, \eqref{ugli} and \eqref{3.3}, we find
     \begin{equation}\label{mu3}
     \begin{split}
       k(x,Sx)=R(x, Sx, x, Sx),\quad k(x,S^{2}x)=R(x, S^{2}x, x,  S^{2}x),\\
       k(u,Su)=\frac{R(u, Su, u, Su)}{1-\cos^{2}\varphi} ,\quad k(u,S^{2}u)=R(u, S^{2}u, u,  S^{2}u).
       \end{split}
 \end{equation}
Taking into account \eqref{mu3}, from \eqref{r+r} we obtain
\begin{equation}\label{mu+mu2}
    \begin{split}
   2(1-\cos^{2}\varphi)k(u,Su)+&k(u,S^{2}u)=k(x,S^{2}x) +2(1-\cos^{2}\varphi)k(x,Sx)\\+&\cos^{2}\varphi (4R_{3}+6R_{4})+4\cos\varphi(R_{5}+R_{6}).
    \end{split}
\end{equation}
The equality \eqref{mu+mu} follows from \eqref{r1-6} and \eqref{mu+mu2}.
   \end{proof}
Next, we express the sectional curvatures of the 2-planes $\{u, Su\}$ and $\{u, S^{2}u\}$ by the sectional curvatures of the 2-planes of the same type, whose angles between basis vectors are constants.
    \begin{thm}
Let $(M, g, S)$ be a manifold with property \eqref{R}. If vectors $x$, $u$, $v$ and $w$ induce $S$-bases and $\{S^{3}x, S^{2}x, Sx, x\}$ is an orthonormal basis, then the sectional curvatures satisfy
\begin{equation}\label{mu-r2}
\begin{split}
   2(1-\cos^{2}\varphi)k(u,Su)+k(u,S^{2}u)&=(1-4\cos^{2}\varphi)\big(2k(x,Sx)+k(x,S^{2}x)\big)  \\& +(2\cos^{2}\varphi+\cos\varphi)\big(\frac{3}{2}k(v,Sv)+k(v,S^{2}v)\big)\\&+(2\cos^{2}\varphi-\cos\varphi)\big(\frac{3}{2}k(w,Sw) +k(w,S^{2}w)\big),
   \end{split}
\end{equation}
where $\angle(u, Su)=\varphi$,  $\angle(v, Sv)=\frac{\pi}{3}$,  $\angle(w, Sw)=\frac{2\pi}{3}$.
\end{thm}
\begin{proof}
In \eqref{mu+mu2} first we substitute $u=v$ and then $u=w$. Successively we get
\begin{align*}
&\frac{3}{2}\big(k(v, Sv)-k(x, Sx)\big)+k(v, S^{2}v)-k(x, S^{2}x)=\frac{1}{2}(2R_{3}+3R_{4})+2(R_{5}+R_{6}),\\
&\frac{3}{2}\big(k(w, Sw)-k(x, Sx)\big)+k(w, S^{2}w)-k(x, S^{2}x)=\frac{1}{2}(2R_{3}+3R_{4})-2(R_{5}+R_{6}).
\end{align*}
We solve the above system and express the tensors $2R_{3}+3R_{4}$ and $R_{5} + R_{6}$ by the curvatures. We substitute these solutions into \eqref{mu+mu2}, and find \eqref{mu-r2}.
 \end{proof}

 The following two statements refer to the manifolds $(M, g, S)$ with property \eqref{R1}, which are a special case of the manifolds considered so far.
\begin{prop}\label{pred_L1}
Let $(M, g, S)$ be a manifold with property \eqref{R1}. If a vector $x$ induces an $S$-basis, then for the sectional curvatures of the basic $2$-planes we have
\begin{equation}\label{3.21*}
\begin{split}
k(x,Sx)&=k(Sx,S^{2}x)=k(S^{2}x,S^{3}x)=k(S^{3}x,x),\\
k(x,S^{2}x)&=S(Sx,S^{3}x)=2(1-\cos^{2}\varphi)k(x,Sx),
\end{split}
\end{equation}
where $\angle(x, Sx)=\varphi.$
\end{prop}

\begin{proof}
If $(M, g, S)$ has the property \eqref{R1}, then it satisfies \eqref{R}. Thus
having in mind \eqref{R1}, \eqref{slRv1} and \eqref{dop4}, we obtain
\begin{equation*}
\begin{split}
  R(x, S^{2}x, x, S^{2}x)=&R(Sx, S^{3}x, Sx, S^{3}x)=2R(x, Sx, x, Sx)\\=2R(x, S^{3}x, x, S^{3}x) &
  2R(Sx, S^{2}x, Sx, S^{2}x)=2R(S^{2}x, S^{3}x, S^{2}x, S^{3}x).
  \end{split}
\end{equation*}
 Then  \eqref{2.1} and \eqref{3.3} imply  \eqref{3.21*}.
\end{proof}

\begin{thm}\label{th6}
Let $(M, g, S)$ be a manifold with property \eqref{R1}. If vectors $x$, $u$ and $v$ induce $S$-bases and $\{S^{3}x, S^{2}x, Sx, x\}$ is an orthonormal basis, then the following equality is valid
\begin{equation}\label{mu-r4}
\begin{split}
   k(u,Su)=\frac{1}{1-\cos^{2}\varphi}(1+2\cos^{2}\varphi-3\cos\varphi)k(x, Sx)+\frac{3\cos\varphi}{2(1-\cos^{2}\varphi)}k(v, Sv),
   \end{split}
\end{equation}
where  $\angle(u, Su)=\varphi$, $\angle(v, Sv)=\frac{\pi}{3}$.
\end{thm}
\begin{proof}
We apply \eqref{R1} to \eqref{slRv1}, \eqref{dop4}, \eqref{dop3} -- \eqref{dop8}, and for the components \eqref{r1-6} we get
\begin{equation}\label{R1-R4}
  R_{1}=2R_{2}=2R_{3}=2R_{4},\quad R_{5}=R_{6}.
\end{equation}
Using the presentation of $R(u,Su, u, Su)$ in Theorem~\ref{th4} and having in mind \eqref{R1-R4} we calculate
\begin{equation}\label{R-L}
\begin{split}
       R(u,Su, u, Su)&=\big((\alpha^{2}+\beta^{2}+\gamma^{2}+\delta^{2})^{2}+2(\alpha\beta+\beta\gamma+\delta\gamma-\alpha\delta)^{2}\big)R(x, Sx, x, Sx)\\&+2(\alpha\beta+\beta\gamma+\delta\gamma-\alpha\delta) R(x, Sx, x, S^{2}x).
       \end{split}
 \end{equation}
 We assume that $u$ is a unit vector and from \eqref{cos-alfa} and \eqref{R-L} it follows
\begin{equation}\label{R-L1}
\begin{split}
       R(u,Su, u, Su)=(1+2\cos^{2}\varphi)R(x, Sx, x, Sx)+2\cos\varphi R(x, Sx, x, S^{2}x).
       \end{split}
 \end{equation}
Then \eqref{mu3} imply
\begin{equation}\label{mu-L1}
\begin{split}
   k(u,Su)=\frac{1+2\cos^{2}\varphi}{1-\cos^{2}\varphi}k(x, Sx)+\frac{2\cos\varphi}{1-\cos^{2}\varphi} R(x, Sx, x, S^{2}x).
   \end{split}
\end{equation}
Now we substitute  $v$ for $u$ in \eqref{mu-L1} and express $R(x, Sx, x, S^{2}x)$ by $k(v,Sv)$ and $k(x,Sx)$. We apply this expression to \eqref{mu-L1} and obtain \eqref{mu-r4}.
 \end{proof}

\section{Conclusion}
In this paper we obtain curvature properties of a 4-dimensional Riemannian manifold $(M, g, S)$ with \eqref{R} associated with a locally conformal Hermitian manifold $(M, g, J=S^{2})$. Furthermore, the curvature properties of the associated manifold $(M, g, J)$ deserve to be studied, also in the case when the manifold belongs to the class $\mathcal{L}_{3}$, according to Grey's classification.
The problem of constructing examples of 4-dimensional almost Einstein manifolds $(M, g, S)$ with \eqref{R} also remains.

\author{Iva Dokuzova\\                                 
    Department of Algebra and Geometry\\University of Plovdiv Paisii Hilendarski\\ 24 Tzar Asen, 4000 Plovdiv, Bulgaria\\
e-mail:dokuzova@uni-plovdiv.bg}

\end{document}